\documentclass[12pt,a4paper]{article}
\usepackage{amssymb,amsmath,amsthm}
\usepackage{graphicx}
\usepackage{hyperref}

\newcommand\cF{{\mathcal F}}

\newcommand\cN{{\mathcal N}}

\newcommand\cT{{\mathcal T}}

\newcommand\ex{\mathrm{ex}}

\newcommand\exa{\mathrm{exa}}

\theoremstyle{plain}
\newtheorem{theorem}{Theorem}[section]

\newtheorem{proposition}[theorem]{Proposition}

\theoremstyle{definition}

\newtheorem{defn}[theorem]{Definition}

\newcommand\cref[1]{Corollary~\ref{cor:#1}}

\def\marrow{{\boldmath {\marginpar[\hfill$\rightarrow \rightarrow$]{$\leftarrow \leftarrow$}}}}

\if10     
\usepackage[mathlines]{lineno}
\newcommand*\patchAmsMathEnvironmentForLineno[1]{%
  \expandafter\let\csname old#1\expandafter\endcsname\csname #1\endcsname
  \expandafter\let\csname oldend#1\expandafter\endcsname\csname end#1\endcsname
  \renewenvironment{#1}%
     {\linenomath\csname old#1\endcsname}%
     {\csname oldend#1\endcsname\endlinenomath}}%
\newcommand*\patchBothAmsMathEnvironmentsForLineno[1]{%
  \patchAmsMathEnvironmentForLineno{#1}%
  \patchAmsMathEnvironmentForLineno{#1*}}%
\AtBeginDocument{%
\patchBothAmsMathEnvironmentsForLineno{equation}%
\patchBothAmsMathEnvironmentsForLineno{align}%
\patchBothAmsMathEnvironmentsForLineno{flalign}%
\patchBothAmsMathEnvironmentsForLineno{alignat}%
\patchBothAmsMathEnvironmentsForLineno{gather}%
\patchBothAmsMathEnvironmentsForLineno{multline}%
}
\linenumbers
\fi

\def\marrow{{\boldmath {\marginpar[\hfill$\rightarrow \rightarrow$]{$\leftarrow \leftarrow$}}}}
\def\gd#1{{\sc GDANI: }{\marrow\sf #1}}

\def\kb#1{{\sc KB: }{\marrow\sf #1}}

\title{On graphs that contain exactly $k$ copies of a subgraph, and a related problem in search theory}

\author{D\'aniel Gerbner$^{a}$,
Bal\'azs Keszegh$^{a,b}$, 
D\'aniel Lenger$^{a}$,
D\'aniel T. Nagy$^{a}$
\\
D\"om\"ot\"or P\'alv\"olgyi$^{b}$, 
Bal\'azs Patk\'os$^{a}$,
M\'at\'e Vizer$^{a,c}$,
G\'abor Wiener$^{c}$
\\
\small $^a$ Alfr\'ed R\'enyi Institute of Mathematics\\
\small $^b$ ELTE Eötvös Lor\'and University\\
\small $^c$ Budapest University of Technology and Economics, \\
\small Department of Computer Science and Information Theory \\
}

\date{}

\begin{document}

\maketitle

\begin{abstract}
    We study $\mathrm{exa}_k(n,F)$, the largest number of edges in an $n$-vertex graph $G$ that contains exactly $k$ copies of a given subgraph $F$. The case $k=0$ is the Tur\'an number $\ex(n,F)$ that is among the most studied parameters in extremal graph theory. We show that for any $F$ and $k$, $\mathrm{exa}_k(n,F)=(1+o(1))\ex(n,F))$ and determine the exact values of $\mathrm{exa}_k(n,K_3)$ and $\mathrm{exa}_1(n,K_r)$ for $n$ large enough.
    We also explore a connection to the following well-known problem in search theory. We are given a graph of order $n$ that consists of an unknown copy of $F$ and some isolated vertices. We can ask pairs of vertices as queries, and the answer tells us whether there is an edge between those vertices. Our goal is to describe the graph using as few queries as possible.
    Aigner and Triesch in 1990 showed that the number of queries needed is at least $\binom{n}{2}-\mathrm{exa}_1(n,F)$. Among other results we show that the number of queries that were answered NO is at least $\binom{n}{2}-\mathrm{exa}_1(n,F)$.
\end{abstract}

\section{Introduction}


In this paper, we study a problem that connects two seemingly distant areas of graph theory. In extremal graph theory, the following questions are among the most fundamental ones. Given a property of graphs, how many edges can a graph $G$ of order $n$ satisfying the property have? The most basic question is due to Tur\'an \cite{T} dealing with the property that $G$ does not contain a given graph $F$ as a subgraph. Here we study the variant where $G$ contains exactly $k$ copies of $F$.

In combinatorial search theory, there is an input unknown to us, and we want to determine either the input itself  or some property of it by asking so-called queries. Our goal is to use as few queries as possible. While the basic question does not assume any structure on the input (see \cite{dh} for details on general search problems), some others do. A natural idea is to assume that the input is a graph, in which case the queries may correspond to edges or subgraphs.

Below we give a more detailed introduction with precise definitions both to the extremal problem and the search problem we study, as well as their connection.

\subsection{Extremal graphs with a given number of copies of some graph}

Tur\'an theory deals with $\ex(n,F)$, the largest number of edges in simple graphs of order $n$  that contain zero copies of a fixed subgraph $F$. Tur\'an \cite{T} proved that $\ex(n,K_{r+1})=|E(T(n,r))|$, where $T(n,r)$ is the
complete $r$-partite graph of order $n$ with each part of order $\lfloor n/r\rfloor$ or $\lceil n/r\rceil$. The famous Erd\H os-Stone-Simonovits theorem \cite{ersi,es} states that if $F$ has chromatic number $r$, then $\ex(n,F)=|E(T(n,r))|+o(n^2)$.

Here we study a generalization, where zero is replaced by other numbers. Let $\cN(H,F)$ denote the number of subgraphs of $H$ isomorphic to $F$.
We start with a very general definition.

\begin{defn}
Let $A$ be a set of non-negative integers, $\cF$  a family of graphs and $n$ a positive integer. Then $\exa_A(n,\cF)$ denotes the largest possible number of edges in a graph $H$ on $n$ vertices, such that $\sum_{F\in\cF} \cN(H,F)\in A$, i.e., the total number of copies of members of $\cF$ contained in $H$ belongs to $A$ (if exists). 
In the case $\cF=\{F\}$, we simply write $\exa_A(n,F)$ instead of $\exa_A(n,\{F\})$ and if $A$ contains only the element $a$, we write $\exa_a(n,\cF)$ instead of $\exa_{\{a\}}(n,\cF)$.
\end{defn}


In this generality, this covers a lot of different topics. First, $\exa_0(n,\cF)=\ex(n,\cF)$. More generally, if $A=\{0,1,\dots,m\}$, then $\exa_A(n,\cF)+1$ is the smallest number of edges that forces more than $m$ copies of $\cF$. The so-called supersaturation phenomenon, that describes at least how many copies of $F$ we can find in a graph of order $n$ with a given number of edges, is among the most studied extensions of Tur\'an theory, see e.g. \cite{ES,rei}.  For example, the earliest result of this type, due to Rademacher (1941, unpublished) states that a graph with $\lfloor n^2/4\rfloor+1$ edges contains at least $\lfloor n/2\rfloor$ triangles. In our language, it says that for $A=\{0,1,\dots,\lfloor n/2\rfloor-1\}$, we have $\exa_A(n,K_3)\le \lfloor n^2/4\rfloor$.
Erd\H os \cite{erdos} asked whether $\exa_0(n,G)=\exa_{\{0,1\}}(n,G)$ holds for all graphs. As pointed out in \cite{kinyai},  this is not true in general, but still might hold for every $n$ and $G$ if $n\ge n_0(G)$.


In this paper, we examine only fixed and finite sets $A$, while $n$ goes to infinity. It is obvious that $\exa_A(n,\cF)=\max \{\exa_{a}(n,\cF): a\in A\}$, thus we can focus on the case when $A$ consists of a single integer $k$.

Let us consider first fixed and finite systems $\cF$. 
In this case $\exa_k(n,\cF)$ and $ex(n,\cF)$ are asymptotically equal.

\begin{proposition}\label{asym}
$\exa_k(n,\cF)=(1+o(1))\ex(n,\cF)$.
\end{proposition}

For graphs with chromatic number at least three, the above result was proved by Aigner and Triesch \cite{aitr} in the case $k=1$.
We provide a proof of this proposition and other related results in Section 2.

\bigskip


Let us turn our attention to the case when $\cF$ contains a graph of an order increasing in $n$. For this setting, we are only aware of results concerning $\exa_1(n,\cF)$, and tipically  all graphs in $\cF$ have order $n$. 

For even values of $n$  let $M_n$ denote the matching on $n$ vertices. Hetyei (Corollary~1.6 in \cite{lov}) showed $\exa_1(n,M_n)=n^2/4$.
Sheehan \cite{shee} proved $\exa_1(n,C_n)=n^2/4+1$ (note that \textit{uniquely Hamiltonian graphs} have been studied by several researchers).
Hendry \cite{hend}, Johann \cite{joh} and Volkmann \cite{vol} studied $\exa_1(n,\cF)$ in the case when $\cF$ consists of all $k$-factors (for some fixed $k$).
Hoffmann and Volkmann \cite{hofvol} considered the same problem for $[1,k]$-factors (i.e. spanning subgraphs with all degrees between 1 and $k$). 
Hoffmann, Sidorowicz, and Volkmann \cite{hofsidvol} studied the largest number of edges in bipartite graphs containing a unique $k$-factor.




\subsection{Search theory problems}

Several other results were obtained by Aigner and Triesch \cite{aitr} in connection with a problem in search theory. Suppose that a family $\cF$ of graphs is fixed. We are given a set $V$ of $n$ vertices, and an unknown graph $F_0$ on this vertex set, which is isomorphic to some $F$ from the family $\cF$ (apart from possible isolated vertices in $F_0$). We can ask a pair of vertices from $V$ as a query, and the answer is YES if there is an edge between those two vertices in $F_0$, and NO otherwise. The goal is to identify $F_0$ using as few queries as possible. Note that it is not enough to determine $F$ itself, e.g., if $\cF=\{C_3,C_4\}$, it is not enough to determine that $F_0$ is isomorphic to, say $C_4$ plus some isolated vertices, but we also need to tell which edges belong to this $C_4$. Let us denote by $L(n,\cF)$ the complexity of this problem, i.e. the number of queries needed to do this in the worst case, for the best algorithm. In the case $\cF=\{F\}$, we write $L(n,F)$ instead of $L(n,\{F\})$.

One can also look at this problem 
as a game between a Questioner and an Adversary. The Questioner asks the queries, and the Adversary answers them arbitrarily, but has to make sure that the YES answers form a graph $F_0$ isomorphic to some $F\in\cF$ at the end (apart from possible isolated vertices in $F_0$). 
The goal of the Questioner is to identify $F_0$ with as few queries as possible, while the Adversary's goal is the opposite, to delay the identification as much as possible. Then if both players play optimally, the game ends after exactly $L(n,\cF)$ queries. 

It is worth mentioning that a related problem has attracted more attention (see e.g. \cite{aa}). Here each query corresponds to a set $D$ of vertices and the answer is YES if and only if $S$ induces an edge of $F_0$. As the queries are more general, the lower bounds for the number of queries in this problem are lower bounds for $L(n,\cF)$, as well. E.g. it was shown in \cite{aa} that $\binom{n}{2}$ queries are needed if $\cF$ is the family of all stars, thus we also have $L(n,\cF)=\binom{n}{2}$.

\medskip 

Aigner and Triesch \cite{aitr} found the following connection between parameters $L(n,\cF)$ and $\exa_1(n,\cF)$. 

\begin{theorem}\label{aigtri}
$L(n,\cF)\ge \binom{n}{2}-\exa_1(n,\cF)$.
\end{theorem}

The basic idea behind the proof is the following. The Adversary answers NO, whenever possible. Note that at the end of the game the pairs that were not asked, together with the pairs that were answered YES, form a graph that contains exactly one copy of a graph from $\cF$. We will prove a strengthening of this result in Section \ref{search} showing that the inequality is quite close to being an equality. 



Aigner and Triesch \cite{aitr} also determined $\exa_1(n,\cF)$ and $L(n,\cF)$ for the following families: the star o order $n$; the family of all trees of order $n$; the complete graph $K_l$ for some $l\le n$; the matching $M_l$ for some $l\le n$. They also gave upper bounds on $L(n,\cF)$ for Hamiltonian paths and cycles (i.e., when $\cF=\{P_n\}$ and when $\cF=\{C_n\}$). 


\subsection{Structure of the paper}



In Section 2 we study $\exa_k(n,F)$ when $F$ is a fixed graph. We prove generalizations of Proposition \ref{asym} and the following exact results.

\begin{theorem}\label{klikk} We have
$\exa_1(n,K_r)=\binom{r}{2}+(r-2)(n-r)+\ex(n-r,K_r)$.
\end{theorem}

\begin{theorem}\label{haromszog}
If $n$ is large enough, we have $\exa_k(n,K_3)=\lfloor (n-1)^2/4\rfloor+k+1$.
\end{theorem}

In Section 3 we study $\exa_1(n,F)$ when $F$ is a complete bipartite graph of order $n$ and show that this is connected to a certain problem in number theory.

In Section 4 we study the search theory problem of determining  $L(n,F)$ and improve the known bounds for bipartite graphs $F$. In Section 4.1 we study variants of this problem and prove generalizations of Theorem \ref{aigtri}.

\section[Fixed bounded size F]{Fixed bounded size $\cF$}

We prove a stronger version of Proposition \ref{asym}.

\begin{proposition}\label{ok} Assume that every graph in $\cF$ is connected, and let $v$ denote the smallest order of the graphs in $\cF$. Then
$(1-2kv/n)\ex(n,\cF)\le \exa_k(n,\cF)\le \ex(n,\cF)+k = \exa_0(n,\cF) + k$.
\end{proposition}

\begin{proof}
The second inequality is obvious, as if we take a graph $G$ with $k$ copies of graphs from $\cF$, then we can obtain an $\cF$-free graph by deleting an edge from each copy.

\noindent In order to prove the first inequality, let $G$ be an $\cF$-free graph  of order $n$ with $m=\ex(n,\cF)$ edges, and let $V_0$ be a set of $kv$ vertices of $G$ of the smallest degrees. By a simple averaging argument, the degree sum of the vertices in $V_0$ is at most $2kvm/n$, thus by deleting $V_0$ from $G$ we obtain a graph $G'$ with at least $(1-2kv/n)\ex(n,\cF)$ edges. Now let $F\in \cF$ be a graph of order $v$, such that $F$ is inclusion minimal in $\cF$, and let $G''$ be the vertex disjoint union of $G'$ and $k$ vertex disjoint copies of $F$. Since all graphs in $\cF$ are connected, $G''$ contains exactly $k$ copies of graphs from $\cF$ and at least $(1-2kv/n)\ex(n,\cF)$ edges, finishing the proof.
\end{proof}

Let us consider now disconnected graphs. It is well-known that if $F$ has components  $F_1,\dots,F_p$, then 
the asymptotics of $\ex(n,F)$ is determined by the component with the largest chromatic number. Interestingly, in our case the asymptotics is given by the component with the smallest chromatic number. More precisely, if $q$ is the smallest and $r$ is the largest chromatic number of the components $F_i$, then the following hold. By the Erd\H os-Stone-Simonovits theorem \cite{ersi,es}  $\ex(n,F)=(1+o(1))(1-\frac{1}{r-1})\binom{n}{2}$, while 
Aigner and Triesch \cite{aitr} showed that $\exa_1(n,F)=(1+o(1))(1-\frac{1}{q-1})\binom{n}{2}$ (actually they proved $L(n,F)=(1+o(1))\frac{1}{q-1}\binom{n}{2}$). Observe that if $q=2$, we only obtain the upper bound $o(n^2)$. Here we improve this latter bound, and also extend it to $\exa_k(n,F)$.

\begin{proposition} Let $F$ be the vertex disjoint union of its connected subgraphs $F_1,\dots,F_p$. Assume $\exa_k(n,F)$ exists, i.e. there exists a graph of order $n$ containing exactly $k$ copies of $F$. Then
$\ex(n,\{F_1,\dots,F_p\})(1-2k|V(F)|/n)\le \exa_k(n,F)\le \ex(n,\{F_1,\dots,F_p\})+O(n)$.
\end{proposition}

\begin{proof}
For the lower bound, we proceed similarly to the proof of Proposition \ref{ok}: we take an $\{F_1,\dots,F_p\}$-free graph of order $n$ with $\ex(n,\{F_1,\dots,F_p\})$ edges, delete  the $k|V(F)|$ vertices of the smallest degree, and add a graph of the smallest possible order containing exactly $k$ copies of $F$  (note that it has at most $k|V(F)|$ vertices).

\noindent For the upper bound, let us consider a graph of order $n$ with $\exa_k(n,F)$ edges that contains exactly $k$ copies of $F$. There are at most $k|V(F)|$ vertices incident to edges of these $k$ copies of $F$, and there are at most $k|V(F)|n=O(n)$ edges incident to these vertices. The rest of the vertices must induce an $\{F_1,\dots,F_p\}$-free graph, thus there are at most $\ex(n,\{F_1,\dots,F_p\})$ edges induced by them, concluding the proof. 
\end{proof}

Note that $\ex(n,\{F_1,\dots,F_p\})\le \ex(n,F_i)$ for every $i\le p$. A conjecture of Erd\H os and Simonovits \cite{erdsim} states that there is an $i\le p$, such that $\ex(n,\{F_1,\dots,F_p\})=\Theta(\ex(n,F_i))$. 
If this conjecture holds, then again only the component with the smallest extremal number matters.

The above proposition gives the asymptotics of $\exa_k(n,F)$ if $\ex(n,F)$ is superlinear, i.e. when no $F_i$ is a tree, otherwise we only obtain an $O(n)$ upper bound. Observe that this does not even necessarily give the order of magnitude, as $\ex(n,\{F_1,\dots,F_p\})$ might be a constant, for example if $F_1=K_2$. Indeed, the value $\exa_k(n,F)$ can also be a constant, e.g. if $F$ is a matching of two edges, then $\exa_1(n,F)=4$ for every $n\ge 4$.

\bigskip

It is easy to see that there is room for improvement in the lower bound of Proposition~\ref{ok}. If the sum of the $kv$ smallest degrees is close to $2kvm/n$, we can try to choose them in such a way that they span some edges. Also, we can often add exactly $k$ copies of $F$ using less than $kv$ vertices. Finally, we might also add some edges between the vertices of the $k$ copies of $F$ and the remaining vertices.

The upper bound can often be improved, as well, since adding an edge to an $\cF$-free graph with $\ex(n,\cF)$ edges often creates several copies of members from $\cF$. 
An example where the upper bound is almost sharp is the star $S_r$ with $r$ leaves. It is obvious that for $n\ge r$, the maximum number of edges of an $S_r$-free graph of order $n$ is $\lfloor n(r-1)/2\rfloor$.
It is also easy to see that for $k<n$ we have $\exa_k(n,S_r)\ge ex(n,S_r)+\lfloor k/2\rfloor -1$  and if both $n(r-1)$ and $k$ are even, then actually $\exa_k(n,S_r)= n(r-1)/2+k/2$. 

Now let us denote by $\zeta(G)$ the largest number $z$, such that it is possible to add a new vertex $v$ together with $z$ edges incident to $v$ to a copy of $G$ without creating another copy of $G$. We obviously have $\zeta(G)\ge \delta(G)-1$ (where $\delta(G)$ is the minimum degree in $G$). For example, $\zeta(K_r)=r-2$. The following proposition is not much harder to prove.

\begin{proposition}\label{trivi} Let $v$ be the order of $F$. Then $$\exa_1(n,F)\le \binom{v}{2}+\zeta(F)(n-v)+\ex(n-v,F).$$

\end{proposition} 

\begin{proof} Let $G$ be a graph of order $n$ with $\exa_1(n,F)$ edges and let us consider the single copy of $F$ in $G$. The vertices of $F$ span at most $\binom{v}{2}$ edges and there are at most $\zeta(F)(n-v)$ edges between the vertices of $F$ and the set of the rest of the vertices $V'$. Finally, we have at most $\ex(n-v,F)$ edges spanned by $V'$.
\end{proof}

Clearly, the first term $\binom{v}{2}$ is sharp only for a clique. Notice that Theorem \ref{klikk} states that for cliques we have equality in Proposition \ref{trivi}. 


\begin{proof}[Proof of Theorem \ref{klikk}] The upper bound follows from Proposition \ref{trivi}. For the lower bound, we take a copy of $K_r$ with vertices $v_1,\dots,v_r$ and a Tur\'an graph on the other $n-r$ vertices with classes $V_1,\dots,V_{r-1}$. Then we connect each vertex in $V_i$ to every vertex of the $K_r$, except for $v_i$ and $v_{i+1}$. This way we constructed a graph $G$ with the desired number of edges. 

Assume there is a copy of $K_r$ in $G$ besides the one on vertices $v_1,\dots,v_r$. We denote this other copy of $K_r$ by $K$. Then $K$ contains $p$ vertices from $\{v_1,\dots,v_{r}\}$ for some $p<r$. We claim that $K$ avoids $p$ classes. Indeed, $K$ avoids the classes with the same indices as its vertices, thus we are done unless $K$ contains $v_r$. Let $q$ be the largest index of a vertex of $K$ such that $v_{q-1}\not\in K$. Then $q>1$ because $K$ is not the original copy of $K_r$. Then $K$ also avoids $V_{q-1}$, thus $K$ avoids at least $p$ classes indeed.
As $K$ contains at most one vertex from each of the remaining classes, there are at most $p+r-1-p=r-1$ vertices in $K$, a contradiction.
\end{proof}

Recall that Theorem \ref{haromszog} states that if $n$ is large enough, then $\exa_k(n,K_3)=\lfloor (n-1)^2/4\rfloor+k+1$.
We use a theorem of Brouwer \cite{br}, that determines the largest number of edges in $K_{r+1}$-free graphs on $n$ vertices that are not $r$-partite, assuming $n\ge 2r+1$. 
For triangles it states that if $H$ is triangle-free and non-bipartite, then $H$ has at most $\lfloor n^2/4\rfloor-\lfloor n/2\rfloor+1=\lfloor (n-1)^2/4\rfloor+1$ edges.


\begin{proof}[Proof of Theorem \ref{haromszog}]
For the lower bound we take the complete bipartite graph $K_{\lfloor \frac{n-1}{2}\rfloor,\lceil \frac{n-1}{2}\rceil}$ 
and an additional vertex connected to one vertex from one class and $k$ vertices from the other class.

For the upper bound, let $G$ be a graph containing exactly $k$ triangles. 

\textbf{Case 1.} There is a vertex $v$ contained in every triangle of $G$. Let $G'$ be a graph obtained by deleting an edge incident to $v$ from each of the $k$ triangles. If $G'$ is not bipartite, then it has at most $\lfloor (n-1)^2/4\rfloor+1$ edges, thus $G$ has at most $\lfloor (n-1)^2/4\rfloor+k+1$ edges and we are done. Thus we may assume that $G'$ is bipartite.

Let now $G''$ be the graph obtained by deleting $v$ from $G$. $G''$ is bipartite as it is a subgraph of $G'$, which is bipartite. Let $v$ be connected to a set $P$ of $p$ vertices in one of the parts and a set $Q$ of $q$ vertices in the other part of $G''$. Then there must be exactly $k$ edges between $P$ and $Q$ (thus $pq\ge k$). It means that $pq-k$ edges are missing there, thus the total number of edges is at most $\lfloor (n-1)^2/4\rfloor +p+q-pq+k \le \lfloor (n-1)^2/4\rfloor +k+1$. In the inequality we used that $(p-1)(q-1)\ge 0$. 


\textbf{Case 2.} There is no vertex contained in every triangle of $G$. 
Let $X$ be the set of vertices contained in some triangle of $G$ and let $x=|X|$. Let $Y=V(G)\setminus X$ and $G'''$ be the subgraph of $G$ induced by $Y$. Observe that for any edge $uv$ in $Y$ the two endvertices are connected to at most $x-2$ vertices from $X$, otherwise their neighborhoods (which are disjoint) would avoid only one vertex of $X$. As this vertex is not contained in every triangle of $G$, there is a triangle (in $X$), such that each of its vertices are connected to $u$ or $v$. Two of its vertices are connected to the same vertex, say $u$, thus $u$ is in a triangle, a contradiction.

We claim that the number of edges $f$ of $G$ incident to the vertices in $Y$ is at most $(n-2)^2/4$. In order to see this, let us choose a maximal matching $M$ in $G'''$ and suppose it has $m$ edges. By the previous observation, for each edge of $M$ there are at most $x-2$ edges going from its two endvertices to $X$. Also, from every other vertex of $Y$ there are at most $x-2$ edges going from this vertex to $X$. Thus we have at most $ (n-x-m)(x-2)$ edges between $X$ and $Y$. 

Claim 12 from \cite{gyk} implies that $Y$ induces at most $m(n-x-m)$ edges. For the sake of completeness, we include here another simple argument showing this. As $M$ is a maximal matching, inside $Y$ there are no edges between vertices not incident to $M$. As $G'''$ is triangle-free, in $Y$ every vertex is connected to at most one endvertex of an edge from $M$ and thus the $n-x-2m$ vertices non-incident to $M$ have degree at most $m$ inside $Y$. Furthermore, the $2m$ vertices spanning $M$ induce a triangle-free graph, thus inducing at most $m^2$ edges. This implies that $Y$ induces at most $(n-x-2m)m+m^2=m(n-x-m)$ edges. 

Altogether, $$f\le (n-x-m)(x-2)+m(n-x-m).$$ After rearranging we obtain $$f\le (n-2)^2/4-(2m-n+2x-2)^2/4\le (n-2)^2/4,$$ as claimed.

This implies $|E(G)|\le (n-2)^2/4+\binom{x}{2}$. As $x\le 3k$, we have that $|E(G)|$ is smaller than $(n-1)^2/4$ for $n$ large enough, and we are done.
%
\end{proof}

We note that the proof implies stability in the sense that if the extremal graph does not have a vertex, such that every triangle is incident to it, then the graph has at most $(n-2)^2/4+O(1)$ edges.


\section{Complete bipartite graphs on \textit{n} vertices}

In this section, we study $\exa_1(n,G)$ if $G$ has $n$ vertices. Therefore, we write $\exa_1(G)$ instead of $\exa_1(n,G)$.


Consider $G=K_{A,B}$, i.e., a complete bipartite graph whose parts have $A$ and $B$ vertices, where $A+B=n$.
We say that $A=\sum_{i=1}^a A_i$ and $B=\sum_{j=1}^b B_j$ is a \emph{unique partition} of $A$ and $B$ if there is no other way to partition the $a+b$ natural numbers $A_i,B_j$ into two parts, whose sum is $A$ and $B$, respectively.
This is meant to also imply that $A_i\ne B_j$, but we allow $A_i=A_{i'}$.
For example, $6=3+3$ and $53=13\times 4+1$ is a unique partition of $6$ and $53$, and so is $6=3+3$ and $53=4\times 13+1$ (indeed, the only way to obtain $6$ as a sum of some of the integers $3,3,13,13,13,13,1$ is $3+3$), but $6=3+3$ and $53=50+3$ is not.
If $A=B$, then $A=\sum_{j=1}^b B_j$ and $B=\sum_{i=1}^a A_i$ counts as the same partition.
For example, $6=3+3$ and $6=2+2+2$ is a unique partition of $6$ and $6$, but $6=3+3$ and $6=3+3$ is not.

Let $mup(A,B)=\max \{a+b\mid$ there is a unique partition $A=\sum_{i=1}^a A_i$ and $B=\sum_{j=1}^b B_j\}$.
Notice that $mup(A,B)$ is well-defined for all $A,B$, except $A=B=1$, as there is at least one unique partition: if $A\le B$, then we can take $B$ and $A$ pieces of 1`s.
We define $mup(1,1)=2$.
Observe that $A+B$ is a trivial upper bound on $mup(A,B)$, since $a\le A$ and $b\le B$, thus $A+1\le mup(A,B)\le A+B$.

\begin{proposition} For any $A,B$ we have  
$\exa_1(K_{A,B})=\binom n2-A-B+mup(A,B)$.
\end{proposition}
\begin{proof}
To prove $\binom n2-A-B+mup(A,B)\le \exa_1(K_{A,B})$, consider a unique partition $A=\sum_{i=1}^a A_i$ and $B=\sum_{j=1}^b B_j$ of $A$ and $B$ into $a+b=mup(A,B)$ parts.
Consider now the $n=A+B$ vertex graph $H$ that is the complement of the vertex disjoint union of $mup(A,B)$ trees  on $A_1,\ldots,A_a,B_1,\ldots,B_b$ vertices, respectively.
Since $\sum (A_i-1)+\sum (B_j-1)=A+B-a-b$, $H$ has $\binom n2-A-B+mup(A,B)$ edges.
It follows from the definition of $mup(A,B)$ that $H$ contains exactly one copy of $K_{A,B}$.

For the proof of the other direction, suppose that $H$ contains at most one copy of $K_{A,B}$ and the number of its edges is maximum among such graphs.
Then the complement of $H$ must be a forest.
With a similar counting as the one above, we obtain $\binom n2-A-B+mup(A,B)\ge \exa_1(K_{A,B})$.
\end{proof}

We are not aware of any results concerning $mup(A,B)$; here 
we prove the following simple bounds. 

\begin{proposition}\label{prop:nuc}
If $n>c$, then
$mup(n,c)= \frac n{\nu}+O_c(1)$, where $\nu$ denotes the smallest \emph{non}-divisor of $c$.
Moreover, in an optimal unique partition, all but $O_c(1)$ parts are equal to $\nu$.
\end{proposition}


\begin{proof}
For the lower bound, leave $c$ undivided (i.e., in one part of size $c$), and divide $n$ into  parts of size $\nu$, except for one part of size at least $c+1$ but otherwise as small as possible.
As this last part is larger than $c$, we have to use it for $n$, and from the size $\nu$ parts we cannot build $c$, as $\nu$ does not divide $c$. Altogether we have $\frac{n}{\nu}-O_c(1)$ parts, as required.

For the upper bound, notice that we can have at most $\frac cd$ parts of size $d$ for any divisor $d$ of $c$.
Indeed, if we had more, then from $\frac cd$ parts of size $d$ we could build $c$, and put another part of size $d$ in $n$, which means that the partition was not unique, by our definition.
Therefore, we can have at most $\sum_{d=1}^{\nu-1} \frac cd\le c(\ln \nu+1)$ such parts of size at most $\nu-1$.
As all other parts have size at least $\nu$, there can be at most $\frac{n+c}\nu$ of them.
In total, we have at most $c(\ln \nu+1) + \frac{n+c}{\nu}=\frac{n}{\nu}+O_c(1)$ parts.

Moreover, if there are at least $\nu$ parts larger than $\nu$, then we can select some of them whose sum is divisible by $\nu$, and replace them with $\nu$-sized parts.
This way we obtain more parts, and the partition is surely still unique, provided we already had at least $c/\nu$ $\nu$-sized parts, since then at least one of these needed to go to $n$, so no $\nu$ can be used to build $c$. On the other hand, if there had been less than $c/\nu$ $\nu$-sized parts, then we would have had at most $(n-O_c(1))/(\nu+1)>n/\nu+O_c(1)$, contradicting the first part of the proposition.
This finishes the proof of the ``moreover'' part of the statement.
\end{proof}

The constants hidden in the $O_c(1)$ term can be probably easily improved further using a bit of number theory.
We can also prove the following, somewhat stronger statement.

\begin{proposition}
If $n>N(c)$, then $mup(n+\nu,c)=mup(n,c)+1$, where $\nu$ is the smallest non-divisor of $c$.
\end{proposition}
\begin{proof}
We use the ``moreover'' part of Proposition \ref{prop:nuc}.
The bound $mup(n+\nu,c)\ge mup(n,c)+1$ follows by adding one more part of size $\nu$ to $n$, while the bound $mup(n+\nu,c)\le mup(n,c)+1$ follows by taking away a part of size $\nu$ from $n+\nu$.  

In the lower bound, if the resulting partition was not unique, then $n+\nu=\sum_{i=1}^{a} A_i=\sum_{i=1}^{a'} A'_i$, where the integers $A_i'$ are chosen from $A_j,B_\ell$. We can assume without loss of generality that $A_1=A_1'=\nu$. Then $n=\sum_{i=2}^{a} A_i=\sum_{i=2}^{a'} A'_i$, thus the original partition of $n$ was not unique.

Similarly, in the upper bound, if the resulting partition was not unique, then $n=\sum_{i=1}^{a} A_i=\sum_{i=1}^{a'} A'_i$, where the integers $A_i'$ are chosen from $A_j,B_\ell$. Let $A_0=A_0'=\nu$, then $n+\nu=\sum_{i=0}^{a} A_i=\sum_{i=0}^{a} A'_i$, thus the original partition of $n+\nu$ was not unique.
\end{proof}

Exact values for some small numbers $c$ and $n$, and a discussion on the problem can be found at \url{https://mathoverflow.net/questions/345548}.

\section{Searching for subgraphs}\label{search}

Let us consider first $L(n,F)$ for a fixed graph $F$. Aigner and Triesch \cite{aitr} showed that if $F$ is connected, then $L(n,F)=(1+o(1))(\binom{n}{2}-ex(n,F))=(1+o(1))\frac{1}{\chi(F)-1}\binom{n}{2}$,  
while if $F$ is disconnected and $q$ is the smallest number among the chromatic numbers of its components, then $L(n,F)=(1+o(1))\frac{1}{q-1}\binom{n}{2}$.

Observe that for graphs with chromatic number two, the above result gives $L(n,F)=\binom{n}{2}-o(n^2)$. This is quite sharp, but we can obtain a better bound on $\binom{n}{2}-L(n,F)$ by exploring the connection of $L(n,F)$ and $\exa_1(n,F)$ further.

\begin{proposition} \label{lando} Let $F$ be a fixed connected graph. Then $\binom{n}{2}\le L(n,F)+\exa_1(n,F)\le \binom{n}{2}+O(n)$. 
Thus, if $F$ is not a tree then $L(n,F)=\binom{n}{2}-(1+o(1))ex(n,F)$.

\end{proposition}

\begin{proof} The first inequality is Theorem \ref{aigtri} which was proved by Aigner and Triesch, we will later also prove a more general version, see Proposition \ref{vesszo}. For the second inequality, let $G$ be a graph of order $n$ containing a unique copy of $F$ with $\exa_1(n,F)$ edges. Let us ask now the edges not in $G$ as queries. If all the answers are NO, then the unique copy of $F$ in $G$ is the one we wanted to find and we are done. Otherwise, if we obtain a YES answer at least once, then by asking all the edges incident to the endpoints of the edge with the YES answer, we can find another edge of $F$ with $O(n)$ further queries. Repeating this process by asking each edge incident to an endpoint of an already found edge of the copy of $F$, we can find the copy using $O(1)O(n)$ queries, since $F$ is fixed and connected.

When $F$ is not a tree, the second statement follows by using Proposition \ref{asym}.
\end{proof}



We remark that if there exists a graph $G$ of order $n$ with maximum degree $\Delta$ containing a unique copy of $F$, then the above proof gives the upper bound $L(n,F)\le \binom{n}{2}-|E(G)|+O(\Delta)$. In particular, if $F$ is a tree of order $t$, it is widely believed that $\lfloor n/(t-1)\rfloor$ copies of $K_{t-1}$ (note that this graph avoids $F$) have $(1+o(1))ex(n,F)$ edges (Erd\H os-S\'os conjecture). This can be easily modified by changing only $O(1)$ edges to have a unique copy of $F$: remove all edges incident to some $t$-tuple of the vertices, and put a copy of $F$ on this $t$-tuple. 
Thus, if the Erd\H os-S\'os conjecture holds, then the second statement in Proposition \ref{lando} follows for trees as well.

\bigskip
Now let us turn our attention to the case when the order of $F$ can grow with $n$.
Aigner \cite{aig} mentions that $L(n,\cF)$ might  always be close to $\binom{n}{2}-\exa_1(n,\cF)$ in some sense. We show that this is not the case. Let us denote by $K_n^-$ the graph that we obtain by deleting a single edge from $K_n$. Then $\exa_1(n,K_n^-)=\binom{n}{2}-1$. Yet also $L(n,K_n^-)=\binom{n}{2}-1$. This follows by applying the Adversary strategy opposite to the one described in the Introduction: if all the answers are YES, we cannot identify $K_n^-$ if there are exist two pairs not yet asked. Note that the Adversary strategy of answering NO as long as possible, would be quite inefficient for $K_n^-$.

\subsection{Variants of the search theory problem}

Here we initiate the study of two different variants of $L(n,\cF)$. In the first case we count only the queries that were answered NO. The goal of Questioner is to finish the algorithm and identify $F\in\cF$ after as few NO answers as possible, while the goal of Adversary is the opposite.
We denote the length of this game (i.e. the number of queries needed in the worst case, for the best algorithm) by $x(n,\cF)$ . 

For the second variant, it is not enough to identify $F\in\cF$, we also have to prove that Questioner actually identified it, she has to ask every edge of it as a query, i.e. after identifying it, he has to ask the edges of $F$ that were not previously asked. We denote  the length of this game by $x'(n,\cF)$. 

\begin{proposition}\label{xx}
If every $F\in\cF$ has the same number $e$ of edges, then we have $x(n,\cF)+e=x'(n,\cF)$.
\end{proposition} 

\begin{proof} The same algorithm is optimal for both problems.
After identifying $F\in\cF$ using $x(n,\cF)$ queries with NO answers, and (say) $m$ YES answers, Questioner can query the $e-m$ other edges of $F$. This shows $x'(n,\cF)\le x(n,\cF)+e$. On the other hand, if Questioner asks $x'(n,\cF)$ queries, identifies $F\in\cF$ and asks all its edges, then she also identified $F$ with $x'(n,\cF)-e$ NO answers, thus $x(n,\cF)\le x'(n,\cF)-e$.
\end{proof}

Our original motivation to study this problem came from a mathematical puzzle posed by G\'asp\'ar \cite{qubit}, where a single player plays a game of \textit{memory}, but without any visual memory. In the original \textit{memory} game, $n$  cards are placed on a table randomly, with their identical backsides up. On the front side, they form $n/2$ identical pairs. The players flip two cards at a time, winning that pair if they are identical. If not, the turn ends, and the next player follows. In the single player version, the goal is to find all the pairs with as few flips as possible. The point of the game is to remember the front side of the cards flipped earlier. However, in G\'asp\'ar's variant the player cannot remember the front side of the cards, only the pairs he already checked. This is equivalent to finding the matching corresponding to the pairs. In this setting, it is natural to ask that the pairs should be found by flipping them up. Thus, it is not enough to identify the matching, but all its edges should be asked, hence here  $x'(n,F)$ is the most natural parameter to study.

To connect the search problem to an extremal problem, we introduce another parameter. Let 
$$\exa'_1(n,\cF)=\max\{|E(G)|-|E(F)|: F\in\cF, |V(G)|=n, $$ $$ \text{ s.t. $G$ contains exactly one copy of $F$ and no copies of other graphs from $\cF$}\}.$$
We are still looking for a graph $G$ of order $n$ that contains a unique graph from $\cF$, but now we only count the edges of $G$ not participating in this copy. Obviously we have $exa'_1(n,\cF)\le \exa_1(n,\cF)$.

\begin{proposition}\label{vesszo} For any $\cF$ and $n$ we have

(\textbf{i}) $\binom{n}{2}-\exa_1(n,\cF)\le x(n,\cF)\le L(n,\cF)\le x'(n,\cF)$,

(\textbf{ii}) $\binom{n}{2}-\exa'_1(n,\cF)\le x'(n,\cF)$.
\end{proposition}

\begin{proof}
The last two inequalities of \textbf{(i)} immediately follow from the definitions. We will prove the remaining two statements together, by describing an adversarial strategy.
During the game at any given time, let $E_y$ be the set of queries answered YES, $E_n$ be the set of queries answered NO, and $E_0$ be the set of pairs not asked. 

The Adversary has the following strategy in both variants: he answers NO to the queries whenever he can, i.e. he answers YES only if otherwise there would be no remaining copy of an $F$ from $\cF$ in $E_y\cup E_0$. First we claim that this strategy is valid, i.e. at the end of the algorithm the Questioner can indeed identify an $F$ from $\cF$. This means that there is exactly one copy of an $F$ from $\cF$ that contains $E_y$ and is contained in $E_y\cup E_0$. Clearly there is at least one such copy, and if there are multiple copies, then the algorithm has not been finished.

We claim that each copy of an $F$ from $\cF$ that is contained in $E_y\cup E_0$ must contain $E_y$. Indeed, if an edge $e\in E_y$ is not contained in such an $F$, then the Adversary could answer NO to that edge. Therefore, at the end of the algorithm, there is only one copy of an $F$ from $\cF$ that is contained in $E_y\cup E_0$. This implies $|E_y\cup E_0|\le \exa_1(n,F)$ and $|E_0|\le \exa'_1(n,F)$. On the other hand, $x(n,F)\ge |E_n|$ and $x'(n,\cF)\ge |E_y|+|E_n|$. Combining these yields the statements.
\end{proof}

Let us return now to the graph $K_n^-$. We have seen that $\exa_1(n,K_n^-)=\binom{n}{2}-1$, thus $x(n,K_n^-)\ge 1$. Indeed, the Adversary strategy of answering NO immediately also shows $x(n,K_n^-)\ge 1$. On the other hand, $x(n,K_n^-)\le 1$, as there cannot be more than one NO answer, since there is only one missing edge.

Let us examine now the families studied by Aigner and Triesch \cite{aitr}. For the matching on $\ell$ vertices and for the complete graph on $\ell$ vertices, they showed $L(n,F)=\binom{n}{2}-\exa_1(n,F)$, which implies $x(n,F)=L(n,F)$, by Proposition \ref{vesszo} and thus also determines $x'(n,F)$, by Proposition \ref{xx}. 

For the star $S_n$ on $n$ vertices, they showed that $\exa_1(n,S_n)=\lfloor (n-1)^2/2\rfloor=\binom{n}{2}-\lfloor n/2\rfloor$, while $L(n,S_n)=\lceil n/2\rceil$, i.e. there is a difference of 1 if $n$ is odd. We show that there is no such difference when considering $x(n,S_n)$, as we have $x(n,S_n)=\lfloor n/2\rfloor$. Indeed, let $n$ be odd, then the Questioner first queries a matching of $\lfloor n/2\rfloor$ edges. If all the answers are NO, the missing vertex is the center of the star, if there is a YES answer to an edge $uv$, the Questioner asks an edge $uw$. If the answer is YES, $u$ is the center, otherwise $v$ is. 

For the family $\cT$ of all trees on $n$ vertices, obviously we have $\exa_1(n,\cT)=n-1$, and Aigner and Triesch \cite{aitr} showed $L(n,\cT)=\binom{n}{2}-1$. Observe that again the difference comes from the YES answers, as there are at most $\binom{n}{2}-(n-1)$ NO answers, showing $x(n,\cT)=\binom{n}{2}-(n-1)$.

\bigskip

These examples may suggest that $x(n,F)=\binom{n}{2}-\exa_1(n,F)$ holds in general. While we have found no single graph $F$ with $x(n,F)>\binom{n}{2}-\exa_1(n,F)$, we expect that there exist such graphs. For families, the situation is different. Let $\cF=\cT\cup \{K_{n-1}\}$. Then obviously $\exa_1(n,\cF)=\binom{n-1}{2}$. On the other hand, assume that the Adversary decides that the unique copy of $F\in \cF$ is a tree. Even if the Questioner knows that, he needs at least $x(n,\cT)=\binom{n}{2}-(n-1)$ queries that are answered NO. 

The difference above clearly comes from the fact that $\exa_1(G)$ is given by one member of $\cF$, while $x(n,\cF)$ is given by the other members. 
However, it is easy to see that $\exa'_1(n,\cF)=0$, which implies $x'(n,\cF)=\binom{n}{2}$. Actually, similarly to the relation between $x(n,F)$ and $\exa_1(n,F)$,  we do not have any example with a strict inequality in \textbf{(ii)} of Proposition \ref{vesszo}.

\bigskip

\textbf{Funding}: Research supported by the National Research, Development and Innovation Office - NKFIH under the grants KH 130371, K 132696,  K 124171, SNN 129364, FK 132060, PD 137779, KKP 139502, and KKP-133819.

Research of B. Keszegh, D. Nagy, and D. P\'alv\"olgyi was also supported by the J\'anos Bolyai Research Fellowship of the Hungarian Academy of Sciences.

Research of B. Keszegh and D. P\'alv\"olgyi was furthermore supported by the New National Excellence Program of the Ministry for Innovation and Technology from the source of the National Research, Development and Innovation Fund, under grants \'UNKP-21-5 and \'UNKP-22-5, and by the National Research, Development and Innovation Office within the framework of the Thematic Excellence Program 2021 - National Research Sub programme TKP2021-NKTA-62: “Artificial intelligence, large networks, data security: mathematical foundation and applications".

Research of G. Wiener was also supported by project no.\ BME-NVA-02, implemented with the support provided by the Ministry of Innovation and Technology of Hungary from the National Research, Development and Innovation Fund, financed under the TKP2021 funding scheme.


\begin{thebibliography}{99}

\bibitem{aig} M. Aigner, Combinatorial Search, Wiley-Teubner Series in Computer
Science (1988)

\bibitem{aitr} M. Aigner, E. Triesch, Searching for subgraphs, in R. Bodendiek (ed.), Contemporary Methods in Graph Theory (1990) 31--45.

\bibitem{aa} N. Alon, V. Asodi, Learning a hidden subgraph. SIAM Journal on Discrete Mathematics, \textbf{18}(4), (2005), 697-712.

\bibitem{br}
A. Brouwer. Some lotto numbers from an extension of Tur\'an's theorem. \textit{Afdeling
Zuivere Wiskunde [Department of Pure Mathematics]}, \textbf{152} (1981).

\bibitem{dh} D.-Z. Du and F. K. Hwang. Combinatorial group testing and its applications, volume 12
of Series on Applied Mathematics. World Scientific Publishing Co., Inc., River Edge,
NJ, second edition, 2000.

\bibitem{erdos} P. Erd\H os, Some of my favourite unsolved problems, A Tribute to Paul Erd\H os. Cambridge University Press, New York (1990), 467–478.

\bibitem{ersi} P. Erd\H os and M. Simonovits. A limit theorem in graph theory. {\it Studia Sci.
Math. Hungar.} {\bf 1} (1966) 51--57.

\bibitem{erdsim} P. Erd\H os, M. Simonovits, Compactness results in extremal graph theory. Combinatorica,
{\bf 2}(3) 275--288, (1982)

\bibitem{ES} P. Erd\H os, M. Simonovits, Supersaturated graphs and hypergraphs. Combinatorica, {\bf 3}(2), 181-192 (1983).

\bibitem{es} P. Erd\H os and A. H. Stone. On the structure of linear graphs. {\it Bulletin of the American Mathematical Society.} {\bf 52}  (1946) 1087--1091.

\bibitem{qubit}
M. E. G\'asp\'ar, 
https://qubit.hu/2019/08/12/esz-ventura-mi-van-ha-memoriaznal-de-nincs-memoriad.

\bibitem{gyk}
E. Győri, B. Keszegh,
On the number of edge-disjoint triangles in $K_4$-free graphs, Combinatorica \textbf{37}(6) (2017), 1113--1124.

\bibitem{hend} G. R. T. Hendry, Maximum graphs with a unique k-factor. J. Combin Theory B
{\bf 37} (1984), 53--63.

\bibitem{hofsidvol} A. Hoffmann, E. Sidorowicz, L. Volkmann, Extremal bipartite graphs with a unique k-factor. Discussiones Mathematicae Graph Theory, {\bf 26}(2), 181-192. (2006)

\bibitem{hofvol} A. Hoffmann, L. Volkmann, On unique $k$-factors and unique $[1, k]$-factors in graphs. Discrete mathematics, {\bf 278}(1-3), 127-138. (2004)

\bibitem{joh} P. Johann, On the structure of graphs with a unique $k$-factor. Journal of Graph Theory, {\bf 35}(4), 227--243. (2000)


\bibitem{lov} L. Lov\'asz, On the structure of factorizable graphs, Acta Math Acad Sci
Hungar {\bf 23} (1972), 179--195.

\bibitem{kinyai}
P. Qiao, X. Zhan,   On a problem of Erdős about graphs whose size is the Turán number plus one. Bulletin of the Australian Mathematical Society, 105(2) (2022), 177--187.

\bibitem{rei}
C. Reiher, The clique density theorem. Annals of Mathematics {\bf 184}(3), (2016), 683--707.

\bibitem{shee} J. Sheehan, Graphs with exactly one hamiltonian circuit, Journal of Graph Theory, \textbf{1}(1), (1977) 37--43.

\bibitem{T}
P. Tur\'an. Egy gr\'afelm\'eleti sz\'els\H o\'ert\'ekfeladatr\'ol. \textit{Mat. Fiz. Lapok}, \textbf{48}, 436--452, 1941.

\bibitem{vol} L. Volkmann, The Maximum Size Of Graphs With A Unique $k$-Factor. Combinatorica, {\bf 24}(3), (2004) 531-540.
\end{thebibliography}
\end{document}